\documentclass[12pt]{article}
\usepackage[utf8]{inputenc}
 \usepackage[left=1in, right=1in, top=1in,bottom = 2in]{geometry}
 \usepackage[shortlabels]{enumitem}
\usepackage{amsmath}
\usepackage{amsthm}
\usepackage{amssymb}
\usepackage{amsfonts}
\usepackage{array}
\usepackage{color}
\usepackage{comment}
\usepackage{dsfont}
\usepackage{easytable}
\usepackage{enumitem}
\usepackage{environ}
\usepackage{fancyhdr}
\usepackage{framed}
\usepackage{graphicx}
\usepackage{hyperref}
\usepackage{listings}
\usepackage{multirow}
\usepackage{tikz-cd}
\usepackage{wasysym}
\usepackage{xcolor}

\newtheorem{theorem}{Theorem}

\newtheorem{lemma}[theorem]{Lemma}
\newtheorem{definition}[theorem]{Definition}
\newtheorem{corollary}[theorem]{Corollary}

\usepackage{arydshln}
\makeatletter
  \renewcommand*\env@matrix[1][*\c@MaxMatrixCols c]{%
    \hskip -\arraycolsep
    \let\@ifnextchar\new@ifnextchar
  \array{#1}}
\makeatother

\newcommand{\Z}{\mathbb{Z}}

\newcommand{\F}{\mathbb{F}}

\providecommand{\keywords}[1]
{
  \small	
  \textbf{\textit{Keywords---}} #1
}

\begin{document}
\author{
  Jos\'e Gustavo Coelho,\\
  \and
  Fabio Enrique Brochero Mart\'{\i}nez\\
}
\date{September 2022}
\title{Counting roots of fully triangular polynomials over finite fields}
\maketitle
\begin{abstract}
Let $\F_q$ be a finite field with $q$ elements, $f \in \F_q[x_1, \dots, x_n]$ a polynomial in $n$ variables and let us denote by $N(f)$ the number of roots of $f$ in $\F_q^n$. 
In this paper we consider the family of fully triangular polynomials, i.e., polynomials  of the form
\begin{equation*}
    f(x_1, \dots, x_n) = a_1 x_1^{d_{1,1}} + a_2 x_1^{d_{1,2}} x_2^{d_{2,2}} + \dots + a_n x_1^{d_{1,n}}\cdots x_n^{d_{n,n}} - b,
\end{equation*}
where $d_{i,j} > 0$ for all $1 \le i \le j \le n$. For these polynomials, we obtain explicit formulas for $N(f)$ when the augmented degree matrix of $f$ is row-equivalent to the augmented degree matrix of a linear polynomial or a quadratic diagonal polynomial. 
\end{abstract}


\section{Introduction}

The classical Markov equation is of the form $x^2+y^2+z^2=3xyz$ and it was studied by Markov in the integer ring. In particular, he showed that the solutions of this equation satisfy a recursive relation, such that the solutions can  be ordered forming a binary tree. Hurwitz considered a generalization of the form 
$$x_1^2+\cdots+x_n^2=a x_1x_2\cdots x_n,$$
and he showed that this equation does not have non-trivial integer solutions when $a>n\ge 3$. 

This equation can be further generalized in many  ways such as changing the number of variables and the exponents, considering the equation in other rings or fields, etc. Many authors have considered generalizations where the equation is over finite fields.

One possible generalization is to consider equations of the form 
\begin{equation}
  a_1 x_1^{m_1} + a_2 x_2^{m_2} + \dots + a_n x_n^{m_n} = b x_1 \cdots x_n,
\end{equation}
where $m_j>0$, $a_j\in \F_q^*$ for all $j=1,\dots, n$ and $b\in \F_q^*$, where $\F_q$ is the finite field with $q$ elements. These equations were studied by Carlitz and Baoulina \cite{carlitz, baoulina}, which determined the number of solutions in $\F_q^n$ under certain restrictions of the exponents.

The generalized Markoff-Hurwitz equation is an equation of the form
\begin{equation}\label{markoff}
  x_1^{m_1} + x_2^{m_2} + \dots + x_n^{m_n} = b x_1^{t_1} x_2^{t_1} \dots x_n^{t_n}.
\end{equation}
where $m_j, t_j>0$ for all $j=1,\dots, n$ and $b\in \F_q^*$. The equation in the case when $m_1=m_2=\dots=m_n=n$ and $t_1=\dots=t_n=1$ defines a hypersurface known as   Calabi-Yau's hypersurface and it has been intensively studied by some authors \cite{calabi yau 1, calabi yau 2}. 
Some results about the number of solutions for the general equation (\ref{markoff}) can be found in the literature; for instance,  the number of solutions over $\F_q$ was calculated by Carlitz in the case when $\gcd (m\sum_{i=1} ^n t_i/m_i -m, q-1)=1$, where $m = m_1 m_2 \cdots m_n$. The case when $\gcd (m\sum_{i=1} ^n t_i/m_i -m, q-1)>1$ was considered by Cao, Jiang and Gao \cite{wei cao, cao 2010}, assuming some arithmetic conditions.

Cao, Wen and Wang \cite{cao 2017} have also determined the number of solutions in $\F_q^n$ for equations of the form
\begin{equation*}
    a_1 x_1^{d_{1,1}} \cdots x_n^{d_{1,n}} + \dots + a_n x_1^{d_{n,1}}\cdots x_n^{d_{n,n}} = 0,
\end{equation*}
where $d_{i,j} > 0$, i.e., all exponents are positive, assuming that the matrix $(d_{i,j})_{i,j}$ is row equivalent to a diagonal  matrix  $D$, when the elements in the diagonal of $D$ are only 1's and 2's.

In this paper we will determine the number of solutions of equations 
\begin{equation*}
    a_1 x_1^{d_{1,1}} + a_2 x_1^{d_{1,2}} x_2^{d_{2,2}} + \dots + a_n x_1^{d_{1,n}}\cdots x_n^{d_{n,n}} = b,
\end{equation*}
where $d_{i,j} > 0$ for all $1 \le i \le j \le n$ and the exponents $d_{i,j}$ satisfy some arithmetical conditions.  In fact, we show sufficient conditions in order to the equation to have, in $(\F_q^*)^n$, the same number of solutions  of a more simple equation. In this case we say that the equations are $*$-equivalent.  Next, we use this equivalence in order to calculate the total number of solutions. 


The remainder of the paper will be organized as follows. In Section 2, we will introduce some preliminary results. In Section 3, we will describe triangular polynomials and relations among them. The main results will be given in Section 4.

\section{Preliminaries}
Let $p$ be a prime number, $q$ a power of $p$ and $\F_q[x_1, \dots, x_n]$ the polynomial ring over $\F_q$ with $n$ variables. For each $D=(d_1, \dots, d_n) \in \Z_{\ge 0}^n$ let us define the monomial $X^D = x_1^{d_1} \cdots x_n^{d_n}$. Given a polynomial $f \in \F_q[x_1, \dots, x_n]$ of the form
\begin{equation}\label{f generico}
    f(x_1, \dots, x_n) = \sum_{j=1}^m a_j X^{D_j},
\end{equation}
where $D_j = (d_{1j}, \dots, d_{nj}) \in \Z_{\ge 0}^n$ and $a_j \ne 0$ for all $j = 1, \dots, m$, we define $N(f)$ as the number of roots of $f(x_1, \dots,$ $ x_n)$ over $\F_q^n$ and $N^*(f)$ as the number of roots over $(\F_q^*)^n$. Let us define the degree matrix of $f$ as $D_f = (D_1^T, \dots, D_m^T)$ and the augmented degree matrix of $f$ as $\tilde{D}_f = ((\tilde{D_1})^T, \dots, (\tilde{D_m})^T)$, where $(\tilde{D_j})^T = (1, D_j)$.

It is well known that the group of multiplicative characters of a finite field is cyclic. Let $\omega$ be a multiplicative character over $\F_q$ with order $q-1$, and therefore $\omega$ is a generator of the multiplicative characters group, i.e., $\widehat{\F_q^*} = \{\omega^k: k = 0, 1, \dots, q - 2\}$. Let us define $Tr$ as the trace function from $\F_q$ to $\F_p$ and $\delta_p$ be a primitive $p$-th complex root of unity. For each integer $0 \le k \le q - 2$, we define the Gauss sum of $\omega^{-k}$ over $\F_q$ as follows:
$$G(k)= \sum_{a \in \F_q^*} \omega(a)^{-k} \delta_p^{Tr(a)}.$$
The following result allows us to express $N^*(f)$ in terms of $\omega$ and the Gauss sums.

\begin{lemma}\label{lema do caractere}
Let $f$ be a polynomial of the form (\ref{f generico}), then
$$N^*(f) = \frac{(q - 1)^n}{q} + \frac{(q - 1)^{n+1-m}}{q}\sum \prod_{j=1}^m \omega(a_j)^{v_j} G(v_j),$$
where the sum is taken over all vectors $v = (v_1, \dots, v_m)$ with $0 \le v_i \le q - 2$ for $i = 1, \dots, m$ such that $\tilde{D}_fv^T \equiv 0 \pmod{q - 1}$.
\end{lemma}
\begin{proof}
See Lemma 2.4 in \cite{wei cao}.
\end{proof}

\begin{definition}
    Two polynomials $f = \sum_j^m a_j X^D_j$ and $g = \sum_j^m a_j X^{D_j^\prime}$ are said to be $*$-equivalent if they have the same coefficient vector $(a_1, \dots, a_m)$ and the congruences $\tilde{D}_f v^T \equiv 0 \pmod{q - 1}$ and $\tilde{D}_g v^T \equiv 0 \pmod{q - 1}$ have the same set of solutions.
\end{definition}

It follows from Lemma \ref{lema do caractere} that if $f$ and $g$ are $*$-equivalent polynomials, then $N^*(f) = N^*(g)$, i.e., they have the same number of roots over $(\F_q^*)^n$.


It is easy to check that the coefficient vectors of two polynomials are equal, but we'd like to know when the linear systems $\tilde{D}_f v^T = 0$ and $\tilde{D}_g v^T = 0$ have the same set of solutions. It can be verified that two matrices $D$ and $E$ with coefficients in $\Z_{q-1}^m$ such that $D v^T \equiv 0 \pmod{q - 1}$ and $E v^T \equiv 0 \pmod{q - 1}$ have the same set of solutions if there is an invertible matrix $M$ over $\Z_{q-1}$ such that $M D = E$, and in this case, we say that $D$ and $E$ are row-equivalent. Hence, if two polynomials $f$ and $g$ have the same coefficient vector and there is an invertible matrix $M$ over $\Z_{q-1}$ such that $M \tilde{D}_f = \tilde{D}_g$, then $f$ and $g$ are $*$-equivalent.

In particular, the elementary row operations are
\begin{enumerate}[$(i)$]
    \item swapping two rows;

    \item adding a multiple of a row to another;

    \item multiplying a row by an element in $\Z_{q-1}^*$;
\end{enumerate}
which can be represented by multiplying invertible matrices, so if we can apply these operations in $\tilde{D}_f$ to obtain $\tilde{D}_g$, the congruence systems have the same solutions. We will use this sufficient criterion to prove $*$-equivalency when needed.


It is worth noting that even though $N^*(f) = N^*(g)$ for two $*$-equivalent polynomials $f$ and $g$, that doesn't mean they have the same set of roots. For instance, the polynomials $f(x,y) = x^2y^3 + x y^2 $ and $g(x,y) = xy + x^3 y^2$ in $\F_5[x,y]$ are $*$-equivalent, but it can be verified that they have distinct sets of roots over $(\F_5^*)^2$.

\section{Triangular polynomials}
Let $f$ be a polynomial in $\F_q[x_1, \dots, x_k]$, and let us define $f_k \in \F_q[x_1, \dots, x_k]$ as
$$f_k (x_1, \dots, x_k) = f(x_1, \dots, x_k, 0, \dots, 0).$$
We say that $f$ and $g$ are totally $*$-equivalent if $f_k$ is $*$-equivalent to $g_k$ for all $1 \le k \le n$. In general, it is not true that $f$ being $*$-equivalent to $g$ implies that $f_k$ is $*$-equivalent to $g_k$ for all $k$. 

Let us introduce a class of polynomials for which a sufficient criterion for total $*$-equivalence can be determined. We say that $f\in \F_q[x_1, \dots, x_n]$ is a triangular polynomial if it is of the form
\begin{equation}\label{triangular superior genérico}
    f(x_1, \dots, x_n) = \sum_{i=1}^n a_j X^{D_j} - b;\quad a_1, \dots, a_n \in \F_q^*, \,b\in \F_q,
\end{equation}
where $D_j = (d_{1,j}, \dots, d_{j,j}, 0, \dots, 0)$, $d_{j,j} > 0$ for all $1 \le j \le n$  and $d_{i,j} \ge 0$ for all $1 \le i < j$. If we additionally have that $d_{i,j} > 0$ for all $1 \leq i \leq j \leq n$, we refer to this polynomial as a \textit{fully} triangular polynomial. 

\begin{lemma}
Let $f,g \in \F_q[x_1, \dots, x_n]$ be two $*$-equivalent triangular polynomials, $M$ the invertible $(n+1)\times(n+1)$ matrix over $\Z_{q-1}$ such that $M \tilde{D}_f = \tilde{D}_g$ and $M_k$ the submatrix obtained from $M$ by picking the first $k+1$ rows and columns. If $M_k$ is invertible, then $f_k$ and $g_k$ are also $*$-equivalent.
\end{lemma}
\begin{proof}
Since $f$ and $g$ are triangular matrices, we can partition their degree matrices and the matrix $M$ into blocks
\begin{equation*}
    \tilde{D}_f = \begin{bmatrix} \tilde{D}_{f_k} & D_1\\ 0 & D_2
    \end{bmatrix}, 
    \qquad
    \tilde{D}_g = 
    \begin{bmatrix}
    \tilde{D}_{g_k} & E_1\\
    0 & E_2
    \end{bmatrix},
    \qquad
    M = 
    \begin{bmatrix}
    M_k & N_1\\
    N_2 & N_3
    \end{bmatrix},
\end{equation*}
such that $\tilde{D}_{g_k}$ and $\tilde{D}_{f_k}$ are $(k+1) \times k$ blocks, $M_k$ is a $(k+1) \times (k+1)$ block and the blocks $D_1, D_2, E_1, E_2, N_1, N_2, N_3$ have appropriate dimensions. From the $*$-equivalency between $f$ and $g$ we know that
\begin{align*}
    \begin{bmatrix}
    \tilde{D}_{g_k} & E_1\\
    0 & E_2
    \end{bmatrix} &= \tilde{D}_g
    = M \tilde{D}_f\\
    &= \begin{bmatrix}
    M_k & N_1\\
    N_2 & N_3
    \end{bmatrix}
    \begin{bmatrix} \tilde{D}_{f_k} & D_1\\ 0 & D_2
    \end{bmatrix}\\
    &= \begin{bmatrix} M_k\tilde{D}_{f_k} & M_k D_1 + N_1 D_2\\ N_2 M_k & N_2 D_1 + N_3 D_2\\
    \end{bmatrix},
\end{align*}
where considering the equality of the upper left block gives us $M_k \tilde{D}_{f_k} = \tilde{D}_{g_k}$, implying that $f_k$ and $g_k$ are $*$-equivalent because $M_k$ is invertible.
\end{proof}

Hence, if $f,g$ are two $*$-equivalent triangular polynomials, with $M \tilde{D_{f}} = \tilde{D_{g}}$ and the submatrices $M_k$ are invertible for every $1 \le k \le n - 1$, then $f$ and $g$ are totally $*$-equivalent. We now present a specific set of operations which always results in transformation matrices that satisfy these conditions.

\begin{lemma}\label{equivalencia triangular superior}
Let $f\in \F_q[x_1, \dots, x_n]$ be a triangular polynomial. Let us denote by $r_1, \dots, r_{n+1}$ the rows in $\tilde{D}_f$ and consider the following invertible row operations:
\begin{enumerate}[(i)]
    \item $r_i \leftarrow c\cdot r_i$, $2 \le i \le n + 1$, $c \in \Z_{q-1}^*$.
    \item $r_j \leftarrow r_j + c \cdot r_i$, $2 \le j < i$, $c \in \Z_{q - 1}$.
\end{enumerate}

Any $*$-equivalency obtained using only these row operations is totally $*$-equivalent. 
\end{lemma}
\begin{proof}
Any matrix $M$ obtained from those operations is of the form
\begin{equation}\label{matriz M}
    M = \begin{bmatrix}
    1 & 0 & 0 & \cdots & 0 & 0 \\
    0 & m_{1,1} & m_{1,2} & \cdots & m_{1, n-1} & m_{1, n}\\
    0 & 0 & m_{2,2} & \cdots & m_{2, n-1} & m_{2, n}\\
    \vdots & \vdots & \vdots & \ddots & \vdots & \vdots\\
    0 & 0 & 0 & \cdots & m_{n-1, n-1} & m_{n-1, n} \\
    0 & 0 & 0 & \cdots & 0 & m_{n, n} \\
    \end{bmatrix},
\end{equation}
where $m_{i,j} \in \Z_{q-1}$ and the elements in the diagonal are invertible. For every $1 \le k < n-1$ the determinant of $M_k$ is $\prod_{i = 1}^{k} m_{i,i}$, which is invertible over $\Z_{q-1}$ and thus every $M_k$ is invertible, making the $*$-equivalency total.

\end{proof}

Although one could believe that all complete equivalences between triangular matrices can be attained using those two operations, this assumption is not correct. For instance, $f = x_1 + x_1^3 x_2^5,\, g = x_1^2 + x_1^4 x_2 \in \F_7[x_1, x_2]$ are totally $*$-equivalent polynomials, i.e., $M \tilde{D}_f = \tilde{D}_g$ where $M$ is the invertible matrix
$$M = \begin{bmatrix} 1 & 0 & 0\\ 1 & 1 & 0\\
0 & 0 & 5\end{bmatrix},$$
but from a straightforward calculation it can be proved that there is no invertible upper triangular matrix $N$ that satisfies $N \tilde{D_f} = \tilde{D_g}$. The following result tells us when a triangular polynomial is totally $*$-equivalent to a diagonal polynomial through the two operations given in Lemma \ref{equivalencia triangular superior}.

\begin{theorem}\label{triangular diagonal equivalence}
Let $f$ be a triangular polynomial of the form (\ref{triangular superior genérico}) and
\begin{equation*}
    g(x_1, \dots, x_n) = a_1 x_1^{e_1} + \dots + a_n x_n^{e_n} -b,
\end{equation*}
be a diagonal polynomial where $e_1, \dots, e_n \in \Z_{> 0}$. Then $f$ is totally $*$-equivalent to $g$ if the following two conditions are true:
\begin{enumerate}[(i)]
    \item for all $1 \le j \le n$ there is a $m_{j,j} \in \Z_{q-1}^*$ such that $d_{j,j} = m_{j,j}e_j$,
    
    \item for all $1 \le i < j \le n$ we have $\gcd (d_{j,j}, q - 1) \mid d_{i,j}$.
\end{enumerate}
\end{theorem}
\begin{proof}
The augmented degree matrices of $f$ and $g$ are
{\footnotesize
$$\tilde{D}_f = \begin{bmatrix}[cccccc:c]
1 & 1 & 1& \cdots & 1 & 1 & 1\\ 
d_{1,1} & d_{1,2} & d_{1,3} &\cdots& d_{1,n-1}& d_{1,n}&0\\
0 & d_{2,2} & d_{2,3} &\cdots& d_{2,n-1}& d_{2,n}&0\\
0 & 0 & d_{3,3} &\cdots& d_{3,n-1} & d_{3,n}&0\\
\vdots & \vdots & \vdots &\ddots & \vdots & \vdots & \vdots\\
0 & 0 & 0 & \cdots & d_{n-1,n-1} & d_{n-1,n}&0\\
0 & 0 & 0 & \cdots & 0 & d_{n,n}&0
\end{bmatrix},\quad \tilde{D}_g = \begin{bmatrix}[cccccc:c]
1 & 1 & 1& \cdots & 1 & 1 & 1\\ 
e_1 & 0 & 0 &\cdots& 0& 0& 0\\
0 & e_2 & 0 &\cdots& 0& 0& 0\\
0 & 0 & e_3 &\cdots& 0 & 0 & 0\\
\vdots & \vdots & \vdots &\ddots & \vdots & \vdots &\vdots \\
0 & 0 & 0 & \cdots & e_{n-1} & 0& 0\\
0 & 0 & 0 & \cdots & 0 & e_{n}& 0
\end{bmatrix},$$}
where the last columns are present only if $b \ne 0$.

Let us suppose that conditions $(i)$ and $(ii)$ are true. Since $(i)$ implies that there is an element $m_{j,j} \in \Z_{q-1}^*$ such that $d_{j,j} = m_{j,j} e_j$, condition $(ii)$ becomes $\gcd (e_j m_{j,j}, q - 1) \mid d_{i,j}$. As $\gcd (m_{j,j}, q - 1) = 1$, that implies condition $(ii)$ is equivalent to $(e_j, q-1) \mid d_{i,j}$, which is in turn equivalent to the existence of $m_{i,j} \in \Z_{q-1}$ such that $d_{i,j} = m_{i,j} e_j$ over $\Z_{q-1}$. We can then use these values of $m_{i,j}$ to construct an invertible matrix $M$ of the form (\ref{matriz M}), such that when we multiply $M$ by $\tilde{D}_g$ gives us
\begin{equation*}
    M \tilde{D}_g = \begin{bmatrix}[cccccc:c]
1 & 1 & 1& \cdots & 1 & 1 & 1\\ 
m_{1,1} e_1 & m_{1,2} e_2 & m_{1,3} e_3&\cdots& m_{1,n-1} e_{n-1} & m_{1,n} e_n &0\\
0 & m_{2,2} e_2 & m_{2,3} e_3 &\cdots& m_{2,n-1} e_{n-1}& m_{2,n} e_n &0\\
0 & 0 & m_{3,3} e_3 &\cdots& m_{3,n-1} e_{n-1} & m_{3,n} e_n &0\\
\vdots & \vdots & \vdots &\ddots & \vdots & \vdots & \vdots\\
0 & 0 & 0 & \cdots & m_{n-1,n-1} e_{n-1}& m_{n-1,n} e_n &0\\
0 & 0 & 0 & \cdots & 0 & m_{n,n} e_n &0
\end{bmatrix},
\end{equation*}
that is equal to $\tilde{D}_f$.
\end{proof}
\vspace{10pt}

For the specific cases where the diagonal polynomials are linear or quadratic this criterion is simpler.
\begin{corollary}\label{corolario triangular diagonal equivalence}
Let $f$ be a triangular polynomial of the form (\ref{triangular superior genérico}) and
\begin{equation*}
    g(x_1, \dots, x_n) = a_1 x_1^{e} + \dots + a_n x_n^{e} -b,
\end{equation*}
be a diagonal polynomial, where $e \in \Z_{> 0}$. 
\begin{enumerate}[a)]
    \item If $e = 1$ and $\gcd (d_{j,j}, q - 1) = 1$ for all $1 \le j \le n$, then $f$ is totally $*$-equivalent to $g$.

    \item If $e = 2$, $q$ is odd, that there exists a $m_{j,j} \in \Z_{q-1}^*$ such that $d_{j,j} = 2m_{j,j}$ for all $1 \le j \le n$ and that $2 \mid d_{i,j}$ for all $1 \le i < j \le n$, then $f$ is $*$-equivalent to $g$.
\end{enumerate}
\end{corollary}
\begin{proof}
    The statements in each item imply the conditions given in Theorem \ref{triangular diagonal equivalence}. In fact,
    \begin{enumerate}[a)]
        \item If $\gcd(d_{j,j}, q-1) = 1$, then condition $(ii)$ is always verified and $d_{j,j}$ is invertible, verifying condition $(i).$ 

        \item The statement that there is an invertible $m_{j,j}$ in $\Z_{q-1}$ such that $d_{j,j} = 2m_{j,j}$ for all $1 \le j \le n$ is, in this case, equivalent to condition $(i)$. Since $m_{j,j}$ is invertible, we have $\gcd(m_{j,j}, q - 1) = 1$, which implies that $\gcd (2, q - 1) = \gcd (2m_{j,j}, q - 1) = \gcd (d_{j,j}, q - 1)$. Considering $q$ odd, we have $\gcd (2, q - 1) = 2$, and the statement $2 \mid d_{i,j}$ for all $1 \le i < j \le n$ is equivalent to condition $(ii)$.
    \end{enumerate}
\end{proof}
We remark that two polynomials being $*$-equivalent does not mean that they have the same number of roots in $\F_q^n$. For instance, the polynomials $f(x,y,z) = 11x^{13} + 5 x^{21} y^{19} + 12 x^{2}y^{3}z^{17}$ and $g(x,y,z) = 11 x + 5 y + 12 z$  are $*$-equivalent in $\F_{31}[x,y,z]$, thus $N^*(f) = 870 = N^*(g)$. However, it can be verified that $N(f) = 1861 \ne 961 = N(g)$.

Let $f$ be a fully triangular polynomial. For any root $(c_1, c_2, \dots, c_n)$ of $f$, if $c_j = 0$ and $j$ is the smallest index that satisfies this condition, then $f(c_1, \dots, c_{j-1}, 0, c_{j+1}^\prime, \dots, c_n^\prime) = 0$ for any $c_{j+1}^\prime, \dots, c_n^\prime \in \F_q$. This is due to the fact that the terms involving the variables $x_{j+1}, \dots, x_n$ vanish, thereby not impacting the value of the polynomial. Thus, by adding over the indices of the first coordinates that are equal to $0$ among the roots, we derive the following identity:
\begin{equation}\label{Nf em termos de N*f}
    N(f) = \begin{cases}
    N^*(f) +\sum_{k = 1}^{n-1} N^*(f_k) q^{n - k - 1},& \text{if }b \ne 0,\\
    q^{n-1} + N^*(f) + \sum_{k = 1}^{n-1} N^*(f_k) q^{n - k - 1},& \text{if }b = 0.\\
    \end{cases}
\end{equation}
Now, let $f$ be totally $*$-equivalent to a polynomial $g$. From Lemma \ref{lema do caractere} we have that $N^*(f_k) = N^*(g_k)$ for all $1 \le k \le n$. Thus,
\begin{equation}\label{Nf em termos de N*g}
    N(f) = \begin{cases}
    N^*(g) +\sum_{k = 1}^{n-1} N^*(g_k) q^{n - k - 1},& \text{if }b \ne 0,\\
    q^{n-1} + N^*(g) + \sum_{k = 1}^{n-1} N^*(g_k) q^{n - k - 1},& \text{if }b = 0.
    \end{cases}
\end{equation}
Therefore, if we know that $f$ is totally $*$-equivalent to a polynomial $g$, and $N^*(g_k)$ is known for any $k$, we can substitute these values into (\ref{Nf em termos de N*g}) to compute $N(f)$. 

For instance, let us consider the polynomials $f,g \in \F_{10007}[x,y]$ given by $f(x,y,z) = x^{1001} + x^{2001}y^{3001} + x^{4001}y^{5001}z^{6001} + 7001$ and $g(x,y,z) = x + y + z + 7001$. By straightforward calculation, we can verify that $f$ is totally $*$-equivalent to $g$, so
\begin{align*}
    N(f) &= N^*(x + y + z + 7001) + N^*(x + y + 7001) + q N^*(x + 7001)\\
    &= 100110031 + 10005 + 10007 \cdot 1\\
    &= 100130043.
\end{align*}

\subsection{Roots with non-zero coordinates for diagonal polynomials}

Let $g$ be the linear polynomial given by
\begin{equation}\label{definição de g linear}
    g(x_1, \dots, x_n) = a_1 x_1 + \dots + a_n x_n -b,
\end{equation}
where $a_1, \dots, a_n \in \F_q^*$ and $b \in \F_q$. In this case the exact value of $N^*(g_k)$ is known, which can be substituted in (\ref{Nf em termos de N*g}) to compute $N(f)$ for any $f$ polynomial $*$-equivalent to $g$.

\begin{lemma}\label{N*(h_k;k) exato}
For a linear polynomial $g$ of the form (\ref{definição de g linear}), the number of roots in $(\F_q^*)^k$ of $g_k$ is
\begin{equation}\label{equação N*(h_k;k) linear}
N^*(g_k) = \begin{cases}
\frac{(q - 1)^k}{q} - \frac{(-1)^k}{q},& \text{if }b \ne 0,\\
\frac{(q - 1)^k}{q} - \frac{(-1)^k}{q} + (-1)^k,& \text{if }b = 0.\\
\end{cases}
\end{equation}

\end{lemma}
\begin{proof}
See Lemma 2 in \cite{cao 2010}.
\end{proof}

In the case when $f$ is totally $*$-equivalent to a quadratic diagonal polynomial
\begin{equation}\label{definição de g}
    g(x_1, \dots, x_n) = a_1 x_1^2 + \dots + a_n x_n^2 - b,
\end{equation}
where $a_1, \dots, a_n \in \F_q^*$ and $b \in \F_q$, we will need a way to compute the number of roots of $g_k$, $1 \le k \le n$ in $\F_q^*$. The following result about quadratic forms is classic.

\begin{theorem}\label{teorema soluções quadraticas}
Let $\F_q$ be a finite field, where $q$ is  odd, and $g$ a polynomial as in (\ref{definição de g}). The number of roots of $g(x_1, \dots, x_n)$ in $\F_q^n$ is

\begin{equation}\label{soluções quadráticas}
N(g) = \begin{cases} 
q^{n-1} - \eta((-1)^{n/2} a_1 \cdots a_n) q^{(n-2)/2},& \text{ if } $n$ \text{ even and } b \ne 0,\\
q^{n-1} + \eta((-1)^{(n-1)/2} b a_1 \cdots a_n) q^{(n-1)/2},& \text{ if } $n$ \text{ odd and } b \ne 0,\\
q^{n-1} + \eta((-1)^{n/2} a_1 \cdots a_n)(q^{n/2} - q^{(n-2)/2}),& \text{ if } $n$ \text{ even and } b = 0,\\
q^{n-1},& \text{ if } $n$ \text{ odd and } b = 0,
\end{cases}
\end{equation}
where $\eta$ is the quadratic multiplicative character in $\F_q$.
\end{theorem}

\begin{proof}
See Theorem 10.5.1 in \cite{gauss jacobi}.
\end{proof}

We notice from this result that the number of roots only depends on the values of $\eta(a_j)$ for $1 \le j \le n$ and $\eta(b)$. 

We also remark that Theorem \ref{soluções quadráticas} let us calculate $N(g_k)$ for any $k$. We will need the following definitions and results in order to  calculate $N^*(g_k)$ for any $k$.

\begin{definition}\label{r(k) e s(k)}
Let $\eta$ be the quadratic character in $\F_q$. For a coefficient vector $(a_1, a_2, \dots, a_n) \in (\F_q^*)^n$ let us define the following functions:
    {\small
\begin{align*}
    r(k) = \# \{1\le j \le k: \eta(a_j) = 1\},\quad
    s(k) = \# \{1\le j \le k: \eta(a_j) = -1\},\quad 1 \le k \le n.
\end{align*}}For simplicity, let us denote $r= r(n),\, s=s(n)$.
\end{definition}


Let us partition the set of roots of $g_k$ in classes $A_{i,j}$ fixing the number $i$ (respectively $j$) of non-zero coordinates of the roots whose corresponding coefficients are squares (respectively non-squares). For any root in $A_{i,j}$, let $\{u_1, \dots, u_{i+j}\}$ be the indices of the $i+j$ non-zero coordinates. Then the non-zero coordinates, arranged in the same order, form a root in $(\F_q^*)^{i+j}$ of the polynomial
$$g_{u_1, \dots, u_{i+j}} = a_{u_1} x_{u_1}^2 + \cdots + a_{u_{i+j}} x_{u_{i+j}}^2 - b.$$

Let $g_{u_1^\prime, \dots, u_{i+j}^\prime}$ be any other polynomial of the same form with the same numbers $i$ and $j$ of square and non-square coefficients. It is easy to construct a bijection between the roots of $g_{u_1, \dots, u_{i+j}}$ and $g_{u_1^\prime, \dots, u_{i+j}^\prime}$ in $\F_q^{i+j}$, and also in $(\F_q^*)^{i+j}$. Thus, $N(g_{u_1, \dots, u_{i+j}})$ and $N^*(g_{u_1, \dots, u_{i+j}})$ depend only on $i$ and $j$. Since the number of roots is the only information that matters to us, we will denote any such polynomial simply by $g_{i,j}$ and the quantities as $N(g_{i,j})$ and $N^*(g_{i,j})$. Thus, the number of roots in each class $A_{i,j}$ is $\binom{r(k)}{i} \binom{s(k)}{j} N^*(g_{i,j})$, and the total number of roots of $g_k$ is

\begin{equation}\label{N em termos de N*}
    N(g_{r(k),s(k)}) = N(g_k) = \sum_{\substack{0 \le i \le r(k) \\ 0 \le j \le s(k)}} \binom{r(k)}{i} \binom{s(k)}{j} N^*(g_{i,j}).
\end{equation}

The following Binomial Inversion Lemma will allow us to obtain an expression of $N^*(g_k)$ in terms of $N(g_{i,j})$.
\begin{lemma}\label{inversão}
Let $G$ be an abelian group and $f:\Z_{\ge 0} \rightarrow G$ a function. Let $F$ be the function defined by $F(r) = \sum_{i=0}^r \binom{ r}{i} f(i)$,
then $f$ can be written in terms of $F$ as
$$f(r) = \sum_{i=0}^r (-1)^{r+i} \binom{r}{i} F(i).$$
\end{lemma}
\begin{proof}
See Section 5.3 in \cite{concrete mathematics}.
\end{proof}

Using Lemma \ref{inversão} twice in (\ref{N em termos de N*}), it follows that

\begin{equation}\label{N*g_krs em termos de Ngij}
    N^*(g_{r(k),s(k)}) = N^*(g_k) = \sum_{i=0}^{r(k)} \sum_{j=0}^{s(k)} (-1)^{r(k)+s(k)+i+j} \binom{r(k)}{i} \binom{s(k)}{j} N(g_{i,j}).
\end{equation}
From Theorem \ref{teorema soluções quadraticas}, and the fact that $\eta(-1) = (-1)^{(q-1)/2}$, we obtain that
{\small
\begin{equation}\label{N(g_ij;i,j) valor explícito}
N(g_{i,j}) = \begin{cases} 
q^{i+j-1} - (-1)^j (-1)^{(q-1)(i+j)/4} q^{(i + j-2)/2},& \text{ if } i+j \text{ even and } b \ne 0,\\
q^{i+j-1} + (-1)^j (-1)^{(q-1)(i+j - 1)/4} \eta(b) q^{(i+j-1)/2},& \text{ if } i+j \text{ odd and } b \ne 0,\\
q^{i+j-1} + (-1)^j (-1)^{(q-1)(i+j)/4}(q^{(i+j)/2} - q^{(i+j-2)/2}), & \text{ if } i+j \text{ even and } b = 0,\\
q^{i+j-1}, & \text{ if } i+j \text{ odd and } b = 0.
\end{cases}
\end{equation}}
By expressing $N^*(g_k)$ in terms of $N(g_{i,j})$, we can use (\ref{N(g_ij;i,j) valor explícito}) to get an explicit value for $N^*(g_k)$, which can be used in (\ref{Nf em termos de N*g}) to determine $N(f)$.

\begin{theorem}\label{soluções não-zero quadraticas}
Let $\F_q$ be a finite field with $q$ an odd number, $g$ a quadratic diagonal polynomial as in (\ref{definição de g}), and $r(k),\,s(k)$ be as in Definition \ref{r(k) e s(k)}. Let us define the complex constants
\begin{align*}
\zeta_1 = (q(-1)^{(q -1)/2})^{1/2} - 1,\quad
\zeta_2 = -(q(-1)^{(q -1)/2})^{1/2} - 1.
\end{align*}
We have that
\begin{enumerate}[a)]
    \item if $b \ne 0$, then
    \begin{equation}\label{equação N*(h_k;k) quadratica b ne 0}
    \begin{split}
    N^*(g_k) =&\, \frac{(q - 1)^{k}}{q}
    - \frac{1}{2q}  ( \zeta_1^{r(k)} \zeta_2^{s(k)} + \zeta_2^{r(k)} \zeta_1^{s(k)}) \\
    &+ \frac{ \eta(b)}{2(q (-1)^{(q-1)/2} )^{1/2}}  ( \zeta_1^{r(k)} \zeta_2^{s(k)} - \zeta_2^{r(k)} \zeta_1^{s(k)}).
\end{split}
\end{equation}

    \item if $b = 0$, then
    \begin{equation}\label{equação N*(h_k;k) quadratica b = 0}
    \begin{split}
    N^*(g_k) =&\, \frac{(q - 1)^{k}}{q}
    - \frac{q - 1}{2q}  ( \zeta_1^{r(k)} \zeta_2^{s(k)} + \zeta_2^{r(k)} \zeta_1^{s(k)}).
\end{split}
\end{equation}
\end{enumerate}

\end{theorem}
\begin{proof}
    Firstly, we remark that both constants $\zeta_1$ and $\zeta_2$ are related to the values of geometric sums. For any positive integer $u$, we have
\begin{equation*}
    \zeta_1^u = (-1)^u \sum_{i=0}^{u} \binom{u}{i} (-1)^i (q(-1)^{(q-1)/2})^{i/2},\quad
    \zeta_2^u = (-1)^u \sum_{i=0}^{u} \binom{u}{i} (q(-1)^{(q-1)/2})^{i/2}.
\end{equation*}

    We will now establish the proof for the scenario when $b \neq 0$ and the analogous case can be reasoned in a similar fashion. Since $b \ne 0$, (\ref{N(g_ij;i,j) valor explícito}) reduces to
\begin{equation*}
    N(g_{i,j}) = 
    \begin{cases} 
    q^{i+j-1} - (-1)^j (-1)^{(q-1)(i+j)/4} q^{(i + j-2)/2},& \text{ if } i+j \text{ even},\\
    q^{i+j-1} + (-1)^j (-1)^{(q-1)(i+j-1)/4} \eta(b) q^{(i+j-1)/2},& \text{ if } i+j \text{ odd},
    \end{cases}
\end{equation*}
which can be rewritten as
\begin{align*}
    N(g_{i,j}) =&\, q^{i+j-1}
    +\frac{(1 + (-1)^{i+j})}{2}\left( - (-1)^j (-1)^{(q-1)(i+j)/4} q^{(i + j-2)/2} \right)\\
    &+\frac{(1 - (-1)^{i+j})}{2}\left( (-1)^j (-1)^{(q-1)(i+j-1)/4} \eta(b) q^{(i+j-1)/2} \right).
\end{align*}
This can be used in (\ref{N*g_krs em termos de Ngij}) to obtain an equation with a right-hand side that can be partitioned into geometric sums, through a straightforward computation, yields our result.
\end{proof}

\section{Main results}
We are going to determine the number of roots of fully triangular polynomials in some cases, starting by the simplest case, when it is totally $*$-equivalent to a linear polynomial. 
\begin{theorem}\label{resultado principal versão linear}
Let $f$ be a fully triangular polynomial as in (\ref{triangular superior genérico}) that is totally $*$-equivalent to a linear polynomial $g$ of form (\ref{definição de g linear}). Then,

\begin{equation*}
    N(f) =\begin{cases}
    
     \frac{q^n - (-1)^n}{q + 1}, & \text{if } b \ne 0,\\
     \frac{2q^n + (-1)^n(q-1)}{q+1}, &\text{if } b = 0.
    \end{cases}
\end{equation*}
\end{theorem}
\begin{proof}
Let us consider the case where $b \ne 0$, because the other is analogous. In this case, (\ref{equação N*(h_k;k) linear}) from Lemma \ref{N*(h_k;k) exato} tells us $N^*(g_k) = \frac{(q - 1)^k}{q} - \frac{(-1)^k}{q}$, and (\ref{Nf em termos de N*g}) implies $N(f) = N^*(g) +\sum_{k = 1}^{n-1} N^*(g_k) q^{n - k - 1}$. Therefore,
\begin{align*}
    N(f) &= \frac{(q - 1)^n}{q} - \frac{(-1)^n}{q} + \sum_{k=1}^n q^{n - k - 1} \left[ \frac{(q - 1)^k}{q} - \frac{(-1)^k}{q}\right]\\
    &= \frac{(q - 1)^n}{q} - \frac{(-1)^n}{q} + q^{n-2}\left[ \sum_{k=1}^n \left(\frac{q - 1}{q}\right)^k + \left(\frac{-1}{q}\right)^k\right]\\
    &=\frac{q^n - (-1)^n}{q + 1}.
\end{align*}

\end{proof}

We remark that notably, the number of roots in $\F_q^n$ of the fully triangular polynomial $f$ is \textit{not} equal to the number of roots of the linear polynomial $g$, which is equal to $q^{n-1}$. We also remark that diagonal polynomials are triangular but not fully triangular, thus any result that requires $f$ to be fully triangular cannot be used on diagonal polynomials.

Notice that the coefficient vector of the fully triangular polynomial does not affect the number of roots of fully triangular polynomials $*$-equivalent to linear polynomials, yielding the following result:

\begin{corollary}\label{corolario principal linear}
    Let $f$ and $h$ be two fully triangular polynomial with $n$ variables of the form (\ref{triangular superior genérico}), such that the constant terms are zero in both of them or non-zero in both of them. If $f$ and $h$ are totally $*$-equivalent to linear polynomials they have the same number of roots. 
\end{corollary}
\begin{proof}
    Even when $f$ and $h$ have different coefficients, Theorem \ref{resultado principal versão linear} implies that the number of roots depends only on the number of variables and also if the constant term is zero or not.
\end{proof}

In the following result we consider the cases when the fully triangular polynomial is totally $*$-equivalent to a quadratic diagonal polynomial.

\begin{theorem}\label{resultado principal versão quadratica}
Let $\F_q$ be a finite field with $q$ an odd number and $f$ be a fully triangular polynomial of the form (\ref{triangular superior genérico}). Let us suppose that $f$ is totally $*$-equivalent to a quadratic diagonal polynomial $g$ of the form (\ref{definição de g}). Let $\zeta_1, \zeta_2$ be the constants  as in Theorem \ref{soluções não-zero quadraticas}. We have that
\begin{enumerate}[a)]
    \item if $b \ne 0$, then
    {\footnotesize
\begin{align*}\label{resultado quando b não 0}
    N(f)
    =&\,  q^{n-2}(q - 1) - \frac{1}{2q} \cdot ( \zeta_1^{r} \zeta_2^{s} + \zeta_2^{r} \zeta_1^{s}) + \frac{ \eta(b)}{2(q (-1)^{(q-1)/2} )^{1/2}} \cdot ( \zeta_1^{r} \zeta_2^{s} - \zeta_2^{r} \zeta_1^{s}) \\
    &+
    \sum_{k=1}^{n-1} q^{n - k - 1} \left( - \frac{1}{2q} \cdot ( \zeta_1^{r(k)} \zeta_2^{s(k)} + \zeta_2^{r(k)} \zeta_1^{s(k)}) + \frac{ \eta(b)}{2(q (-1)^{(q-1)/2} )^{1/2}} \cdot ( \zeta_1^{r(k)} \zeta_2^{s(k)} - \zeta_2^{r(k)} \zeta_1^{s(k)}) \right);
\end{align*}}

    \item if $b = 0$, then
    \begin{align*}
    N(f) =&\, q^{n-1} + q^{n-2}(q - 1) + \left( \frac{q - 1}{2q}\right) (\zeta_1^{r} \zeta_2^{s} + \zeta_2^{r} \zeta_1^{s})\\
    &+ \left( \frac{q - 1}{2}\right) \sum_{k=1}^{n-1} q^{n - k - 2}  (\zeta_1^{r(k)} \zeta_2^{s(k)} + \zeta_2^{r(k)} \zeta_1^{s(k)}).
\end{align*}
\end{enumerate}

\end{theorem}

\begin{proof}

Let us prove the result in the case when $b \ne 0$, because the other case is analogous. In this case, (\ref{equação N*(h_k;k) quadratica b ne 0}) in Theorem \ref{N(g_ij;i,j) valor explícito} implies
\begin{equation*}
    \begin{split}
    N^*(g_k) =&\, \frac{(q - 1)^{k}}{q}
    - \frac{1}{2q}  ( \zeta_1^{r(k)} \zeta_2^{s(k)} + \zeta_2^{r(k)} \zeta_1^{s(k)}) \\
    &+ \frac{ \eta(b)}{2(q (-1)^{(q-1)/2} )^{1/2}}  ( \zeta_1^{r(k)} \zeta_2^{s(k)} - \zeta_2^{r(k)} \zeta_1^{s(k)}),
\end{split}
\end{equation*}
and from (\ref{Nf em termos de N*g}) we have $N(f) = N^*(g) +\sum_{k = 1}^{n-1} N^*(g_k) q^{n - k - 1}$. Therefore,
{\footnotesize\begin{align*}
    N(f)
    =&\,  \frac{(q-1)^n}{q} - \frac{1}{2q} \cdot ( \zeta_1^{r} \zeta_2^{s} + \zeta_2^{r} \zeta_1^{s}) + \frac{ \eta(b)}{2(q (-1)^{(q-1)/2} )^{1/2}} \cdot ( \zeta_1^{r} \zeta_2^{s} - \zeta_2^{r} \zeta_1^{s}) \\
    &+
    \sum_{k=1}^{n-1} q^{n - k - 1} \left( \frac{(q-1)^k}{q} - \frac{1}{2q} \cdot ( \zeta_1^{r(k)} \zeta_2^{s(k)} + \zeta_2^{r(k)} \zeta_1^{s(k)}) + \frac{ \eta(b)}{2(q (-1)^{(q-1)/2} )^{1/2}} \cdot ( \zeta_1^{r(k)} \zeta_2^{s(k)} - \zeta_2^{r(k)} \zeta_1^{s(k)}) \right),
\end{align*}}
which simplifies to obtain the desired result.
\end{proof}

Notice that the specific values in the vector coefficient $(a_1, a_2, \dots, a_n, -b)$ do not matter. In fact, to determine the number of roots is which of the coefficients $a_j$'s and $b$ are squares or not. Thus we have the following result:

\begin{corollary}\label{corolario principal quadratico}
    Let $\F_q$ be a finite field with $q$ an odd number. Let $f$ and $h$ be two fully triangular polynomial with $n$ variables of the form (\ref{triangular superior genérico}) with coefficient vectors $(a_1, a_2, \dots, a_n, -b)$ and $(c_1, c_2, \dots, c_n, - d)$ respectively. Let us suppose that $\eta(a_1) = \eta(c_1),\dots,\eta(a_n) = \eta(c_n),\, \eta(b) = \eta(d)$. If $f$ and $h$ are totally $*$-equivalent to quadratic diagonal polynomials, they have the same number of roots. 
\end{corollary}
\begin{proof}
    The values of $r(k)$ and $s(k)$ for $1 \le k \le n$ will be the same for both polynomials. Hence from Theorem \ref{corolario principal quadratico} they both have the same number of roots in $\F_q^n$.
\end{proof}

Then for every choice of which coefficients in the coefficient vector we have fixed values for $r(k),s(k)$ for $1 \le k \le n$. We can substitute these values in Theorem \ref{teorema soluções quadraticas} to find a closed expression for the number of roots. We will do this for two specific cases.

\begin{corollary}\label{caso tudo quadrado}
    Let $\F_q$ be a finite field with $q$ an odd number and $f$ be a fully triangular polynomial of the form (\ref{triangular superior genérico}) such that the coefficients $a_1, \dots, a_n$ are either all squares, or are all non squares in $\F_q$. Let us suppose that $f$ is totally $*$-equivalent to a quadratic diagonal polynomial $g$ of the form (\ref{definição de g}). We have that

    \begin{enumerate}[a)]
    \item if $b \ne 0$, then
\begin{align*}
    N(f)
    =&\,  q^{n-2}(q - 1) - \frac{1}{2q} \cdot \left( \frac{\zeta_1^{n+1} - (q - 1) \zeta_1^n - \zeta_1 q^{n-1}}{\zeta_1 - q} + \frac{\zeta_2^{n+1} - (q - 1) \zeta_2^n - \zeta_2 q^{n-1}}{\zeta_2 - q} \right) \\
    &+\frac{ \varepsilon \eta(b)}{2(q (-1)^{(q-1)/2} )^{1/2}} \cdot \left( \frac{\zeta_1^{n+1} - (q - 1) \zeta_1^n - \zeta_1 q^{n-1}}{\zeta_1 - q} - \frac{\zeta_2^{n+1} - (q - 1) \zeta_2^n - \zeta_2 q^{n-1}}{\zeta_2 - q}  \right) ;
\end{align*}
    where 
    $$\varepsilon = \begin{cases}
        1, & \text{ if } a_1, \dots, a_n \text{ are squares}, \\
        -1, & \text{ if } a_1, \dots, a_n \text{ are non squares}. 
    \end{cases}$$ 

    \item if $b = 0$, then
    \begin{align*}
    N(f) =&\, 2q^{n-1} - q^{n-2}\\& + \frac{q-1}{2q}\left(\frac{\zeta_1^{n+1} - (q-1)\zeta_1^n - \zeta_1q^{n-1}}{\zeta_1 - q} + \frac{\zeta_2^{n+1} - (q-1)\zeta_2^n - \zeta_2q^{n-1}}{\zeta_2 - q} \right),
\end{align*}
in both cases.
\end{enumerate}
\end{corollary}
\begin{proof}
    In the case when $a_1, \dots, a_n$ are squares we have $r(k) = k,\, s(k) = 0$ for $1 \le k \le n$. Substituting into the expressions in Theorem \ref{resultado principal versão quadratica} and computing the geometric sums yields the result.

    For the case when none of the coefficients $a_1, \dots, a_n$ is a square, we can multiply $f$ by a non square coefficient $a$ to obtain a polynomial with exactly the same roots, but whose coefficients $a a_1, \dots, a a_n$ are squares, and the constant term $ab$ is such that $\eta(ab) = - \eta(b)$. Thus in the case when $b = 0$ the number of roots is exactly the same as the all squares case, and in the case $b \ne 0$ the expression is essentially the same with a couple signs changed.
\end{proof}

\end{document}